\documentclass[10pt]{article}
\usepackage{amsmath,amssymb,amsthm}
\usepackage{graphicx,color}
\usepackage[margin=1in]{geometry}
\usepackage[affil-it]{authblk}

\newtheorem{proposition}{Proposition}[section]

\theoremstyle{definition}

\theoremstyle{remark}
\newtheorem{Remark}{Remark}

\newcommand{\m}{\scalebox{0.75}[1.0]{$-$}}

\newmuskip\pFqskip
\mathchardef\pFcomma=\mathcode`, 

\newcommand*\pFq[5]{%
  \begingroup
  \begingroup\lccode`~=`,
    \lowercase{\endgroup\def~}{\pFcomma\mkern\pFqskip}%
  \mathcode`,=\string"8000
  {}_{#1}F_{#2}\biggl[\genfrac..{0pt}{}{#3}{#4};#5\biggr]%
  \endgroup
}

\numberwithin{equation}{section}
\allowdisplaybreaks

\title{Convolution identities for Dunkl orthogonal polynomials \\ from the $\mathfrak{osp}(1|2)$ Lie superalgebra}
\author[1]{Erik Koelink}
\author[2]{Jean-Michel Lemay}
\author[2]{Luc Vinet}
\affil[1]{IMAPP, Radboud Universiteit, P.O. Box 9010, 6500 GL Nijmegen, The Netherlands}
\affil[2]{Centre de Recherches Math\'ematiques, Universit\'e de Montr\'eal, C.P. 6128,\protect\\ Succ. Centre-ville, Montr\'eal, QC, Canada, H3C 3J7} 
\date{}

\begin{document}
\maketitle

\begin{abstract}
 
 \noindent New convolution identities for orthogonal polynomials belonging to the $q=-1$ analog of the Askey-scheme are obtained. A specialization of the Chihara polynomials will play a central role as the eigenfunctions of a special element of the Lie superalgebra $\mathfrak{osp}(1|2)$ in the positive discrete series representation. Using the Clebsch-Gordan coefficients, a convolution identity for the Specialized Chihara, the dual \m1 Hahn and the Big \m1 Jacobi polynomials is found. Using the Racah coefficients, a convolution identity for the Big \m1 Jacobi and the Bannai-Ito polynomials is found. Finally, these results are applied to construct a bilinear generating function for the Big \m1 Jacobi polynomials.   
 
\end{abstract}

\section{Introduction}

In \cite{Granovskii1993}, Granovskii and Zhedanov proposed an approach to obtain convolution identities for orthogonal polynomials of the Askey-scheme through algebraic methods. The main idea is to study a self-adjoint element of a Lie algebra which corresponds to a recurrence operator diagonalized by orthogonal polynomials in a suitable representation. In the tensor product of representations, the Clebsch-Gordan decomposition and the Racah recoupling can then be used to relate polynomial eigenfunctions in two different bases to arrive at convolution identities. Van der Jeugt \cite{VanDerJeugt1997} expanded on this idea to obtain generalizations of some classical convolution identities for the Laguerre and Hermite polynomials. One of the authors then joined Van der Jeugt \cite{Koelink1998} to exploit this approach further and obtain convolution identities for the Meixner-Pollaczek, the Hahn and the Jacobi polynomials and their descendants with $\mathfrak{su}(1,1)$ as the underlying Lie algebra and also for the Al-Salam Chihara, $q$-Racah and Askey-Wilson polynomials using the quantized analog $U_q(\mathfrak{su}(1,1))$. Two subsequent papers \cite{VanderJeugt1998,Koelink1999} extended this work and derived generating functions and Poisson kernels for some involved polynomials by using differential realizations of the discrete series representations of $\mathfrak{su}(1,1)$ and its $q$-generalization.     

The main goal of this paper is to use this construction to obtain convolution identities for orthogonal polynomials belonging to the Bannai-Ito scheme of \m1 orthogonal polynomials \cite{Bannai1984,Vinet2011,Tsujimoto2012,Vinet2012, Genest2013a,Tsujimoto2013,Genest2014b}. These polynomials arise as the $q=-1$ limits of families belonging to the Askey tableau of $q$-orthogonal polynomials \cite{Koekoek2010}. More precisely, most of its polynomials are defined by $q\to-1$ limits of the Askey-Wilson polynomials and its descendants. The \m1 orthogonal polynomials are eigenfunctions of Dunkl operators \cite{Vinet2011a} which involve the reflexion operator $R$ defined by $R f(x)=f(-x)$ \cite{Dunkl1988,Dunkl2006}. For this reason, they are also called Dunkl orthogonal polynomials. The first example of such polynomials was introduced by Bannai and Ito as a $q\to-1$ limit of the $q$-Racah polynomials in the classification of a category of association scheme \cite{Bannai1984}. They have since been fully characterized \cite{Tsujimoto2012} and have appeared  in various context : superintegrable systems \cite{Genest2013,Genest2014,Genest2015a} and the transport of quantum information \cite{Vinet2012a,Vinet2012b,Tsujimoto2013} for example. The Lie superalgebra $\mathfrak{osp}(1|2)$, sometimes referred to as $sl_{-1}(2)$, has been found to provide a fruitful algebraic underpinning for a number of \m1 polynomials \cite{Tsujimoto2011,Bergeron2016,Vinet2017}. In particular, the Clebsch-Gordan coefficients of $\mathfrak{osp}(1|2)$ can be expressed in terms of dual $-1$ Hahn polynomials \cite{Genest2013b} and the Racah coefficients in terms of Bannai-Ito polynomials \cite{Genest2014a}. These two results are essential ingredients of the main results of this paper : the convolution identities given in propositions \ref{convid1} and \ref{convid2}. The former relates the Specialized Chihara, the dual \m1 Hahn and the Big \m1 Jacobi polynomials and the latter connects the Big \m1 Jacobi and the Bannai-Ito polynomials. These results can also be interpreted in another interesting way. They give connection coefficients for different two-variables polynomials orthogonal with respect to the same measure. This is an interesting feature as the extension to multiple variables of the Bannai-Ito scheme is in its early stages \cite{Genest2015,Lemay2018,DeBie2019}. It is also quite remarkable to have a framework relating so many Dunkl orthogonal polynomials. As additional results, we obtain a generating function for the Specialized Chihara polynomials and follow an approach similar to the one used in \cite{VanderJeugt1998} to obtain a bilinear generating function for the Big \m1 Jacobi polynomials. The discussion in section 5 indicates that the results of \cite{Koelink1999} can be generalized to Lie superalgebra representations; the complexity of the outcome will however increase considerably. 

The paper is structured as follows. The properties of the relevant \m1 orthogonal polynomials are surveyed in section 2 and section 3 is dedicated to a review of the superalgebra $\mathfrak{osp}(1|2)$ and its Clebsch-Gordan and Racah coefficients. We proceed in section 4 to the construction of convolution identities. A self-adjoint element of $\mathfrak{osp}(1|2)$ is introduced and its generalized eigenvectors in the positive discrete series representation are obtained. Looking at the tensor product of irreducible representations, the Clebsch-Gordan coefficients are used to construct a first convolution identity. Next, the three-fold tensor product and the Racah coefficients are considered to obtain a second convolution identity. In section 5, we present a first application of these results and derive a bilinear generating function for the Big \m1 Jacobi polynomials. Finally, some closing remarks are given in the conclusion.

\section{Review of Dunkl orthogononal polynomials}

  The results we present in this paper have the notable feature of connecting various orthogonal polynomials from the Bannai-Ito scheme together. The families involved are the Specialized Chihara polynomials $P_n(\lambda;\mu,\gamma)$, the Big \m1 Jacobi polynomials $J_{n}(x;a,b,c)$, the dual \m1 Hahn polynomials $R_n(x;\eta,\xi,N)$ and finally the Bannai-Ito polynomials $B_n(x;\rho_1,\rho_2,r_1,r_2)$. We review in this section some of their properties while etablishing the notation that will be used throughout this paper.     

  \subsection{Specialized Chihara polynomials}
    A one-parameter extension of the generalized Hermite polynomials was introduced in \cite{Genest2014b} as a special case of the Chihara polynomials which both sit in the $q=-1$ analog of the Askey scheme. We present some of their properties here with a different normalization and a different notation. For simplicity, we name them the Specialized Chihara polynomials and denote them by $P_n(\lambda;\mu,\gamma)=P_n(\lambda)$. They satisfy the 3-term recurrence relation
    \begin{align} \label{OpeghRR}
    \lambda P_n(\lambda) = [n+1]_\mu^{1/2} P_{n+1}(\lambda) + \gamma(-1)^n P_{n}(\lambda) + [n]_\mu^{1/2} P_{n-1}(\lambda)
    \end{align}
    where $[n]_\mu = n +(1-(-1)^n)\mu$ denotes the $\mu$-number. The Specialized Chihara polynomials can be expressed in terms of Laguerre polynomials in the following way :
    \begin{align} 
    \begin{aligned} \label{ExplicitOpegh}
    P_{2n}(\lambda;\mu,\gamma) &=  (-1)^n \sqrt{\frac{n!\ \Gamma(\mu+\tfrac12)}{\Gamma(n+\mu+\tfrac12)}}\ L_n^{(\mu-\frac12)}\left( \frac{\lambda^2-\gamma^2}{2} \right), \\
    P_{2n+1}(\lambda;\mu,\gamma) &=  (-1)^n \sqrt{\frac{n!\ \Gamma(\mu+\tfrac32)}{\Gamma(n+\mu+\tfrac32)}} \left(\frac{\lambda-\gamma}{\sqrt{2\mu+1}}\right) L_n^{(\mu+\frac12)}\left( \frac{\lambda^2-\gamma^2}{2} \right) ,
    \end{aligned}
    \end{align}
    where the Laguerre polynomials are given in terms of the usual hypergeometric function
    \begin{align} 
    L_n^{(\alpha)}(x) = \frac{(\alpha+1)_n}{n!}\ \pFq{1}{1}{-n}{\alpha+1}{x}
    \end{align}
    with $(a)_n = a(a+1)\dots(a+n-1)$ denoting the Pochhammer symbol. The Specialized Chihara polynomials satisfy the orthogonality relation
    \begin{align} \label{OpeghOR}
    \int_{F} P_n(\lambda)P_m(\lambda) w(\lambda,\mu,\gamma)d\lambda = 2\Gamma(\mu+\tfrac12)\delta_{n,m}
    \end{align}
    where $F=  (-\infty,-|\gamma|~) \cup (~|\gamma|,\infty)$ and the weight function is given by
    \begin{align} \label{OpeghW}
    w(\lambda,\mu,\gamma) = \text{sign}(\lambda)(\lambda+\gamma)\left(\frac{\lambda^2-\gamma^2}{2}\right)^{\mu-\frac12}e^{-\left(\frac{\lambda^2-\gamma^2}{2}\right)}.
    \end{align}
    This result can easily be verified from the orthogonality of the Chihara polynomials or from that of the Laguerre polynomials. They satisfy also a specialization of the differential-difference equation obeyed by the Chihara polynomials.

  \subsection{The Big -1 Jacobi polynomials}
    We now review some of the properties of the Big \m1 Jacobi polynomials which shall be needed in the following. Denoted by $J_{n}(x;a,b,c)$, these polynomials are also part of the $q=-1$ analog of the Askey scheme and were introduced in \cite{Vinet2012} as a $q\to-1$ limit of the Big $q$-Jacobi polynomials. They are defined by
    \begin{align}
    \label{Jacobi}
    J_{n}(x;a,b,c)=
    \begin{cases}
    \pFq{2}{1}{-\frac{n}{2},\frac{n+a+b+2}{2}}{\frac{a+1}{2}}{\frac{1-x^2}{1-c^2}}+\frac{n(1-x)}{(1+c)(a+1)}\; \pFq{2}{1}{1-\frac{n}{2}, \frac{n+a+b+2}{2}}{\frac{a+3}{2}}{\frac{1-x^2}{1-c^2}}, & \text{$n$ even},
    \\[3mm]
    \pFq{2}{1}{-\frac{n-1}{2}, \frac{n+a+b+1}{2}}{\frac{a+1}{2}}{\frac{1-x^2}{1-c^2}}-\frac{(n+a+b+1)(1-x)}{(1+c)(a+1)} \; \pFq{2}{1}{-\frac{n-1}{2}, \frac{n+a+b+3}{2}}{\frac{a+3}{2}}{\frac{1-x^2}{1-c^2}},& \text{$n$ odd},
    \end{cases}
    \end{align}
    where ${}_{2}F_{1}$ is the standard Gauss hypergeometric function. We shall simply write $J_{n}(x)$ instead of $J_{n}(x;a,b,c)$ when the parameters are clear from the context. They satisfy the recurrence relation
    \begin{align*}
      x\,J_{n}(x)=A_{n}\, J_{n+1}(x)+(1-A_{n}-C_{n})\,J_{n}(x)+ C_{n}\,J_{n-1}(x),
    \end{align*}
    with coefficients
    \begin{align*}
    A_{n}=
      \begin{cases}
      \frac{(n+a+1)(c+1)}{2n+a+b+2}, & \text{$n$ even},\\[0.5em]
      \frac{(1-c)(n+a+b+1)}{2n+a+b+2}, & \text{$n$ odd},
      \end{cases}
    \qquad 
    C_{n}=
      \begin{cases}
      \frac{n(1-c)}{2n+a+b}, & \text{$n$ even},\\[0.5em]
      \frac{(n+b)(1+c)}{2n+a+b}, & \text{$n$ odd}.
      \end{cases}
    \end{align*}
    It can be seen that for $a,b>-1$ and $|c|\neq 1$ the polynomials $J_{n}(x)$ are positive-definite. The orthogonality relation of the Big \m1 Jacobi polynomials is different for $|c|<1$ and $|c|>1$. In what follows, we only need the polynomials for the first case. For $|c|<1$, one has 
    \begin{align}
      \label{Ortho-1}
      \int_{\mathcal{C}}J_{n}(x;a,b,c)\,J_{m}(x;a,b,c) \;\omega(x;a,b,c)\;\mathrm{d}x=\left[\frac{(1-c^2)^{\frac{a+b+2}{2}}}{(1+c)}\right] h_{n}(a,b)\,\delta_{nm},
    \end{align}
    where the interval is $\mathcal{C}=(-1,-|c|~)\cup(~|c|,1)$ and the weight function reads
    \begin{align}
      \label{Weight-1}
      \omega(x;a,b,c)=\text{sign}(x)\,(1+x)\,(x-c)\,(x^2-c^2)^{\frac{b-1}{2}}\,(1-x^2)^{\frac{a-1}{2}}.
    \end{align}
    The normalization factor $h_{n}$ is given by
    \begin{align}
    \label{hn}
    h_{n}(a,b)=
      \begin{cases}
      \frac{2\;\Gamma\left(\frac{n+b+1}{2}\right)\Gamma\left(\frac{n+a+3}{2}\right) \left(\frac{n}{2}\right)!}{(n+a+1)\;\Gamma\left(\frac{n+a+b+2}{2}\right)\left(\frac{a+1}{2}\right)_{\frac{n}{2}}^2}, & \text{$n$ even}, \\[3mm]
      \frac{(n+a+b+1)\;\Gamma\left(\frac{n+b+2}{2}\right)\Gamma \left(\frac{n+a+2}{2}\right) \left(\frac{n-1}{2}\right)!}{2\,\Gamma\left(\frac{n+a+b+3}{2}\right) \left(\frac{a+1}{2}\right)_{\frac{n+1}{2}}^2}, & \text{$n$ odd}.
      \end{cases}
    \end{align}
    The orthogonality relation for $|c|>1$ and a difference equation can be found in \cite{Vinet2012}.

  \subsection{The dual -1 Hahn polynomials}
    We now introduce a third family of orthogonal polynomials. The dual \m1 Hahn polynomials, denoted by $R_n(x;\eta,\xi,N)$ or $R_n(x)$, depends on two real parameters $\eta, \xi$ and on an integer parameter $N$. They have been introduced in \cite{Tsujimoto2013} as a $q\to-1$ limit of the dual $q$-Hahn polynomials. They have found applications in the transport of quantum information along spin chains \cite{Vinet2012a} and, of importance here, they have also been shown to arise as the Clebsch-Gordan coefficients of the Lie superalgebra $\mathfrak{osp}(1|2)$ \cite{Genest2013b}. They satisfy the 3-term recurrence relation 
    \begin{align}
    x R_n(x) = R_{n+1}(x) + \big((-1)^{n+1} (2\xi +(-1)^N 2\eta)-1\big) R_n(x) + 4[n]_\xi [N-n+1]_\eta R_{n-1}(x).
    \end{align}
    They can be expressed as follows in terms of hypergeometric series. For $N$ even, we have
    \begin{align}
    R_n(x;\eta,\xi,N) =
    \begin{cases}
    16^{\frac{n}{2}} \left(-\tfrac{N}{2}\right)_{\!\frac{n}{2}} \left(\tfrac{1-2\eta-N}{2}\right)_{\!\frac{n}{2}} \pFq{3}{2}{-\tfrac{n}{2},\delta+\tfrac{x+1}{4},\delta-\tfrac{x+1}{4}}{-\tfrac{N}{2},\tfrac{1-2\eta-N}{2}}{1}, \quad &\!\!\text{$n$ even,} \\[1em]
    16^{\frac{n-1}{2}} \left(1\!-\!\tfrac{N}{2}\right)_{\!\frac{n-1}{2}} \left(\tfrac{1-2\eta-N}{2}\right)_{\!\frac{n-1}{2}}(x+2\eta+2\xi+1)~ \pFq{3}{2}{-\tfrac{n-1}{2},\delta+\tfrac{x+1}{4},\delta-\tfrac{x+1}{4}}{1-\tfrac{N}{2},\tfrac{1-2\eta-N}{2}}{1}, &\text{$n$ odd,}
    \end{cases}
    \end{align}
    where $\delta = -\tfrac{\eta +\xi +N}{2}$ and, for $N$ odd, we have
    \begin{align}
    R_n(x;\eta,\xi,N) =
    \begin{cases}
    16^{\frac{n}{2}} \left(\tfrac{1-N}{2}\right)_{\!\frac{n}{2}} \left(\tfrac{2\xi+1}{2}\right)_{\!\frac{n}{2}} \pFq{3}{2}{-\tfrac{n}{2},\delta+\tfrac{x+1}{4},\delta-\tfrac{x+1}{4}}{\tfrac{1-N}{2},\tfrac{2\xi+1}{2}}{1}, \quad &\text{$n$ even,} \\[1em]
    16^{\frac{n-1}{2}} \left(\tfrac{1-N}{2}\right)_{\!\frac{n-1}{2}} \left(\tfrac{2\xi+3}{2}\right)_{\!\frac{n-1}{2}}(x+2\xi-2\eta+1)~ \pFq{3}{2}{-\tfrac{n-1}{2},\delta+\tfrac{x+1}{4},\delta-\tfrac{x+1}{4}}{\tfrac{1-N}{2},\tfrac{2\xi+3}{2}}{1}, \quad &\text{$n$ odd,}
    \end{cases}
    \end{align}
    where $\delta = \tfrac{\eta +\xi +1}{2}$. The dual \m1 Hahn polynomials are orthogonal with respect to the discrete measure 
    \begin{align}
    \sum_{s=0}^N \varpi_s(\eta,\xi,N) R_n(y_s;\eta,\xi,N)R_m(y_s,\eta,\xi,N) = \nu_n(\eta,\xi,N) \delta_{n,m}
    \end{align}
    with the following grid points : 
    \begin{align}
    y_s = 
    \begin{cases}
    (-1)^s(2s-2\eta-2\xi-2N-1), \quad &\text{$N$ even,} \\[0.5em]     
    (-1)^s(2s+2\eta+2\xi+1),    \quad &\text{$N$ odd.}        
    \end{cases}
    \end{align}
    For $N$ even, the weights and normalization factors are given by
    \begin{align}
    \varpi_s(\eta,\xi,N) = 
    \begin{cases}
      (-1)^{\frac{s}{2}}\dfrac{\left(-\frac{N}{2}\right)_{\frac{s}{2}} \left(-\frac{N}{2}-\eta +\frac{1}{2}\right)_{\frac{s}{2}} (-N-\eta -\xi )_{\frac{s}{2}}}{\frac{s}{2}! \left(-\frac{N}{2}-\xi +\frac{1}{2}\right)_{\frac{s}{2}} \left(-\frac{N}{2}-\eta -\xi \right)_{\frac{s}{2}}}, \quad &\text{$s$ even,} \\[1.5em]   
      (-1)^{\frac{s-1}{2}}\dfrac{ \left(-\frac{N}{2}\right)_{\frac{s+1}{2}} \left(-\frac{N}{2}-\eta +\frac{1}{2}\right)_{\frac{s-1}{2}} (-N-\eta -\xi )_{\frac{s-1}{2}}}{\frac{s-1}{2}! \left(-\frac{N}{2}-\xi +\frac{1}{2}\right)_{\frac{s-1}{2}} \left(-\frac{N}{2}-\eta -\xi \right)_{\frac{s+1}{2}}}, \quad &\text{$s$ odd,}
    \end{cases}
    \end{align}
    \begin{align}
    \nu_n(\eta,\xi,N) =
    \begin{cases}
      16^n~ \frac{n}{2}! \left(-\frac{N}{2}\right)_{\frac{n}{2}} \left(\xi +\frac{1}{2}\right)_{\frac{n}{2}} \left(\frac{-N-2 \eta +1}{2}\right)_{\frac{n}{2}} \frac{(-N-\eta -\xi )_{\frac{N}{2}}}{\left(\frac{-N-2 \xi +1}{2}\right)_{\frac{N}{2}}}, \quad &\text{$n$ even,} \\[1.5em]
      -16^n~ \frac{n-1}{2}! \left(-\frac{N}{2}\right)_{\frac{n+1}{2}} \left(\xi +\frac{1}{2}\right)_{\frac{n+1}{2}} \left(\frac{-N-2 \eta +1}{2}\right)_{\frac{n-1}{2}}\frac{(-N-\eta -\xi )_{\frac{N}{2}}}{\left(\frac{-N-2 \xi +1}{2}\right)_{\frac{N}{2}}}, \quad &\text{$n$ odd,} 
    \end{cases}
    \end{align}
    whileas for $N$ odd, the weights and normalization factors are given by
    \begin{align}
    \varpi_s(\eta,\xi,N) = 
    \begin{cases}
      (-1)^{\frac{s}{2}}\dfrac{\left(\frac{1-N}{2}\right)_{\frac{s}{2}} \left(\xi +\frac{1}{2}\right)_{\frac{s}{2}} (\eta +\xi +1)_{\frac{s}{2}}}{\frac{s}{2}! \left(\eta +\frac{1}{2}\right)_{\frac{s}{2}} \left(\frac{N+3}{2}+\eta +\xi\right)_{\frac{s}{2}}}, \quad &\text{$s$ even,} \\[1.5em]   
      (-1)^{\frac{s-1}{2}}\dfrac{\left(\frac{1-N}{2}\right)_{\frac{s-1}{2}} \left(\xi +\frac{1}{2}\right)_{\frac{s+1}{2}} (\eta +\xi +1)_{\frac{s-1}{2}}}{\frac{s-1}{2}! \left(\eta +\frac{1}{2}\right)_{\frac{s+1}{2}} \left(\frac{N+3}{2}+\eta +\xi\right)_{\frac{s-1}{2}}}, \quad &\text{$s$ odd,}
    \end{cases}
    \end{align}
    \begin{align}
    \nu_n(\eta,\xi,N) =
    \begin{cases}
      16^n~ \frac{n}{2}! \left(\frac{1-N}{2}\right)_{\frac{n}{2}} \left(\xi +\frac{1}{2}\right)_{\frac{n}{2}} \left(-\frac{N}{2}-\eta \right)_{\frac{n}{2}}\frac{(\eta +\xi +1)_{\frac{N+1}{2}}}{\left(\eta +\frac12\right)_{\frac{N+1}{2}}}, \quad &\text{$n$ even,} \\[1.5em]
      -16^n~ \frac{n-1}{2}! \left(\frac{1-N}{2}\right)_{\frac{n-1}{2}} \left(\xi +\frac{1}{2}\right)_{\frac{n+1}{2}} \left(-\frac{N}{2}-\eta \right)_{\frac{n+1}{2}}\frac{(\eta +\xi +1)_{\frac{N+1}{2}}}{\left(\eta +\frac12\right)_{\frac{N+1}{2}}}, \quad &\text{$n$ odd.} 
    \end{cases}
    \end{align}
    For future convenience, we introduce 
    \begin{align}
    z_s = (-1)^{s+N+1} (2s+2\eta+2\xi+1), 
    \qquad 
    \rho_s(\eta,\xi,N) = 
    \begin{cases} 
      \varpi_{N-s}(\eta,\xi,N), \quad &\text{$N$ even,}\\
      \varpi_{s}(\eta,\xi,N), \quad &\text{$N$ odd}
    \end{cases}
    \end{align}
    which corresponds to the grid points and weights of the dual \m1 Hahn polynomials with the indices reversed when $N$ is even but not when $N$ is odd.

  \subsection{The Bannai-Ito polynomials}
    We finally present a last family of orthogonal polynomials called the Bannai-Ito polynomials. They were originally discovered by Bannai and Ito \cite{Bannai1984} in their classification of $P-$ and $Q-$ polynomials association scheme which are in correspondance with orthogonal polynomials satisfying the Leonard duality property. In this original setting, they were observed to be $q\to-1$ limit of the $q$-Racah polynomials. They have also been shown to correspond to a $q\to-1$ limit of the Askey-Wilson polynomials \cite{Tsujimoto2012}. The Bannai-Ito polynomials occur in a Bochner-type theorem for first order Dunkl difference operators \cite{Tsujimoto2012} that has them at the top of the $q=-1$ analog of the $q$-Askey tableau. Of  particular relevance to this paper, they are the Racah coefficients for the Lie superalgebra $\mathfrak{osp}(1|2)$ \cite{Genest2014a}.

    The monic Bannai-Ito polynomials $B_n(x;\rho_1,\rho_2,r_1,r_2)$, or $B_n(x)$ for short, depend on 4 parameters $\rho_1,\rho_2,r_1,r_2$ and the linear combination
    \begin{align}
    g=\rho_1+\rho_2-r_1-r_2.
    \end{align}
    They are symmetric with respect to the $\mathbb{Z}_2 \times \mathbb{Z}_2$ group transformations generated by $\rho_1 \leftrightarrow \rho_2$ and $r_1 \leftrightarrow r_2$. 
    Throughout this section, it will be convenient to write integers as follows
    \begin{align} \label{paritynotation}
    n=2n_e+n_p, \qquad n_p\in\{0,1\}, \quad n, n_e\in\mathbb{N}.           
    \end{align}
    The Bannai-Ito polynomials can be defined in terms of two hypergeometric functions 
    \begin{align} \label{hypergeo-expr}
    \frac{1}{\eta_n}B_n(x;\rho_1,\rho_2,r_1,r_2) =\ &\pFq{4}{3}{\scriptstyle -n_e, n_e+g+1, x-r_1+\frac12, -x-r_1+\frac12}{\scriptstyle 1-r_1-r_2, \rho_1-r_1+\frac12, \rho_2-r_1+\frac12}{1} \\[0.5em]
    &+ \frac{(-1)^n(n_e\!+\!n_p\!+\!g n_p)(x\!-\!r_1\!+\!\tfrac12)}{(\rho_1-r_1+\tfrac12)(\rho_2-r_1+\tfrac12)} \pFq{4}{3}{\scriptstyle -n_e-n_p+1, n_e+n_p+g+1, x-r_1+\frac32, -x-r_1+\frac12}{\scriptstyle 1-r_1-r_2, \rho_1-r_1+\frac32, \rho_2-r_1+\frac32}{1} \notag
    \end{align}
    with the normalization coefficients
    \begin{align} \label{eta}
    \eta_n = (-1)^{n} \frac{(\rho_1-r_1+\tfrac12)_{n_e+n_p}(\rho_2-r_1+\tfrac12)_{n_e+n_p}(1-r_1-r_2)_{n_e}}{(n_e+g+1)_{n_e+n_p}}.
    \end{align}
    It is also possible to express the Bannai-Ito polynomials as a linear combination of two Wilson polynomials \cite{Tsujimoto2012}. 

    The $B_n(x)$ satisfy the 3-term recurrence relation
    \begin{gather}
    \label{Recurrence-Relation}
    xB_n(x) = B_{n+1}(x)+(\rho_1-A_n-C_n)B_{n}(x)+A_{n-1}C_{n}B_{n-1}(x),
    \end{gather}
    with the initial conditions $B_{-1}(x)=0$ and $B_{0}(x)=1$.
    The recurrence coef\/f\/icients $A_n$ and $C_n$ are given by
    \begin{gather}\label{Coeff-An}
    \begin{aligned}
    A_n&=
    \begin{cases}
    \dfrac{(n+2\rho_1-2r_1+1)(n+2\rho_1-2r_2+1)}{4(n+g+1)}, & n~\text{even},
    \vspace{1mm}\\
    \dfrac{(n+2g+1)(n+2\rho_1+2\rho_2+1)}{4(n+g+1)}, & n~\text{odd},
    \end{cases}
    \\[1em] 
    &C_n =
    \begin{cases}
    -\dfrac{n(n-2r_1-2r_2)}{4(n+g)}, & n~\text{even},
    \vspace{1mm}\\
    -\dfrac{(n+2\rho_2-2r_2)(n+2\rho_2-2r_1)}{4(n+g)}, & n~\text{odd}.
    \end{cases}
    \end{aligned}
    \end{gather}
    Favard's theorem states that these polynomials will be orthogonal only if they satisfy the positivity conditions $u_n = A_{n-1}C_n >0$. Since this cannot be achieved for all $n\in\mathbb{N}$, the parameters must verify a truncation condition of the form 
    \begin{align} \label{truncationcondition}
    u_0=u_{N+1}=0.
    \end{align}
    The integer $N$ is called the truncation parameter.

    If these conditions are fulf\/illed, the Bannai-Ito polynomials $B_{n}(x)$ satisfy the discrete
    orthogonality relation
    \begin{gather}
    \sum_{k=0}^{N}w_{k}B_{n}(x_k)B_{m}(x_k)=h_n\delta_{nm},
    \end{gather}
    with respect to a positive set of weights~$w_k$.
    The orthogonality grid ~$x_k$ corresponds to the simple roots of the characteristic polynomial $B_{N+1}(x)$.
    The explicit formulas for the weight function $w_{k}$ and the grid points~$x_k$ depend on the
    parity of $N$ and more explicitly on the realization of the truncation condition~$u_{N+1}=0$.

    If $N$ is even, it follows from~\eqref{Coeff-An} that the condition $u_{N+1}=0$ is tantamount to one of the following requirements associated to all possible values of $j$ and $\ell$ :
    \begin{align} \label{Tcond1}
    i)~ r_j - \rho_\ell = \frac{N+1}{2}, \qquad j,\ell \in\{1,2\}.
    \end{align}
    Note that the four possibilities coming from the choices of $j$ and $\ell$ are equivalent since the polynomials $B_n(x)$ are invariant under the exchanges $\rho_1\leftrightarrow\rho_2$ and $r_1\leftrightarrow r_2$.

    If $N$ is odd, it follows from~\eqref{Coeff-An} that the condition $u_{N+1}=0$ is equivalent to
    one of the following restrictions:
    \begin{gather} \label{Tcond2}
    ii)~\rho_1+\rho_2=-\frac{N+1}{2},
    \qquad
    iii)~r_1+r_2=\frac{N+1}{2},
    \qquad
    iv)~\rho_1+\rho_2-r_1-r_2=-\frac{N+1}{2}.
    \end{gather}
    In this paper, we shall only be concerned with the truncation conditions $r_2-\rho_1 = \frac{N+1}{2}$ when $N$ is even and $\rho_1+\rho_2=-\frac{N+1}{2}$ when $N$ is odd.
    In these cases, the grid points have the expression
    \begin{gather}
    \label{gridi}
    x_k=(-1)^{k}(k/2+\rho_1+1/4)-1/4,
    \end{gather}
    for $k=0,\dots,N$ and using \eqref{paritynotation} the weights take the form
    \begin{gather}
    \label{Ortho-Weight}
    w_{k}=\frac{(-1)^{k}}{k_e!}
    \frac{(\rho_1-r_1+1/2)_{k_e+k_p}(\rho_1-r_2+1/2)_{k_e+k_p}(\rho_1+\rho_2+1)_{k_e}(2\rho_1+1)_{k_e}}
    {(\rho_1+r_1+1/2)_{k_e+k_p}(\rho_1+r_2+1/2)_{k_e+k_p}(\rho_1-\rho_2+1)_{k_e}},
    \end{gather}
    where $(a)_{n}=a(a+1)\cdots(a+n-1)$ denotes the Pochhammer symbol. 
    When $N$ is even, the normalization factors are given by
    \begin{align} \label{Ortho-Norm}
    h_{n} = \frac{n_e! N_e! (1+2\rho_1)_{N_e} (1+\rho_1+\rho_2)_{n_e} (1+n_e+g)_{N_e-n_e} (\frac12+\rho_1-r_1)_{n_e+n_p} (\frac12+\rho_2-r_1)_{n_e+n_p} }{ (N_e-n_e-n_p)! (\frac12+\rho_1+r_1)_{N_e-n_e} (\frac12+n_e+n_p+\rho_2-r_2)_{N_e-n_e-n_p} (1+n+g)_{n_e+n_p}^2 }
    \end{align}
    and, for $N$ odd, they are instead given by
    \begin{align} \label{Ortho-Norm2}
    h_{n} = \frac{ n_e! N_e! (1+2\rho_1)_{N_e+1} (1\!-\!r_1\!-\!r_2)_{n_e} (1\!+\!n_e\!+\!g)_{N_e+1-n_e} (\frac12\!+\!\rho_1\!-\!r_1)_{n_e+n_p} (\frac12\!+\!\rho_1\!-\!r_2)_{n_e+n_p} }{ (N_e\!-\!n_e)! (\frac12+\rho_1+r_1)_{N_e+1-n_e-n_p} (\frac12+n_e+n_p+\rho_2-r_2)_{N_e+1-n_e-n_p} (1+n+g)_{n_e+n_p}^2 }.
    \end{align}

    The Bannai-Ito polynomials also verify a difference equation. It was shown in \cite{Tsujimoto2012} that in fact they diagonalize the most general first order Dunkl difference operator with orthogonal polynomials as eigenfunctions.

\section{The $\mathfrak{osp}(1|2)$ Lie Superalgebra}

    This section describes the key entity upon which our study rests namely the Lie superalgebra $\mathfrak{osp}(1|2)$.
    This superalgebra possesses one even generator $J_0$ and two odd generators $J_\pm$. The $\mathbb{Z}_2$-grading will be encoded with the help of an involution operator $R$. The defining relations are
    \begin{align}
    [J_0,J_\pm]=\pm J_\pm, \quad \{J_+,J_-\}=2J_0, \quad [R,J_0]=\{R,J_\pm\}=0, \quad R^2=1
    \end{align}
    where $[A,B]=AB-BA$ and $\{A,B\}=AB+BA$ are the usual commutator and anticommutator.
    There is a Casimir operator 
    \begin{align}
    Q = (J_+J_- -J_0+\tfrac12)R
    \end{align}
    which commutes with all the generators. Moreover, $\mathfrak{osp}(1|2)$ also admits a Hopf algebra structure. In particular, there is an algebra morphism called the coproduct defined on the universal enveloping algebra $\Delta : U(\mathfrak{osp}(1|2)) \to U(\mathfrak{osp}(1|2)) \otimes U(\mathfrak{osp}(1|2))$ which acts as follows 
    \begin{align} \label{coproduct}
    \Delta(J_\pm) = J_\pm \otimes R + 1\otimes J_\pm, \quad \Delta(J_0) = J_0 \otimes1 + 1\otimes J_0, \quad \Delta(R)= R\otimes R
    \end{align}
    and a $*$-structure given by $J_\pm^* = J_\mp, J_0^*=J_0$ and $R^*=R$. 
    The Hilbert space $\ell^2(\mathbb{Z_+})$ equipped with the orthonormal basis $e_n^{(\mu,\epsilon)}, n=0,1,\dots$ supports irreducible representations $(\mu,\epsilon)$ of $\mathfrak{osp}(1|2)$ indexed by two parameters $\mu>0$ and $\epsilon=\pm1$ :
    \begin{align} 
    \begin{aligned} \label{osp12Action}
    J_0 e_{n}^{(\mu,\epsilon)} = (n+\mu+\tfrac12)e_{n}^{(\mu,\epsilon)}, \qquad Re_{n}^{(\mu,\epsilon)} = \epsilon(-1)^ne_{n}^{(\mu,\epsilon)},\\
    J_+e_{n}^{(\mu,\epsilon)} = [n+1]_\mu^{1/2}e_{n+1}^{(\mu,\epsilon)}, \qquad J_-e_{n}^{(\mu,\epsilon)} = [n]_\mu^{1/2}e_{n-1}^{(\mu,\epsilon)}
    \end{aligned}
    \end{align}
    where $[n]_\mu = n + (1-(-1)^n)\mu$ denotes again the $\mu$-number.
    It is called the positive discrete series representation and the decomposition into irreducible components of the tensor product of two such representations is given by
    \begin{align} \label{osp12Decomp}
    (\mu_1,\epsilon_1) \otimes (\mu_2,\epsilon_2) = \bigoplus_{j=0}^\infty \left(\mu_1+\mu_2+\tfrac12+j,\ \epsilon_1\epsilon_2(-1)^j\right).
    \end{align}
    This Clebsch-Gordan decomposition series implies that there is a unitary transformation between the direct product and direct sum bases of the representations involved :
    \begin{align} \label{CGDecomp}
    e_N^{(\mu_{12},\epsilon_{12})} = \sum_{n_1+n_2 = N+j} C_{n_1,n_2}^{N,j} e_{n_1}^{(\mu_1,\epsilon_1)}\otimes e_{n_2}^{(\mu_2,\epsilon_2)}
    \end{align}
    where
    \begin{align} \label{decomp_param}
    \mu_{12} = \mu_1 + \mu_2 + j + \tfrac12, \quad \epsilon_{12} = \epsilon_1\epsilon_2(-1)^j, \quad j=0,1,2,\dots
    \end{align}
    The Clebsch-Gordan coefficients of $\mathfrak{osp}(1|2)$ are given in terms of the dual \m1 Hahn polynomials \cite{Genest2013b} by  
    \begin{align} \label{CGC}
    C_{n_1,n_2}^{N,j} = (-1)^{\phi(n_1,n_2,j)} \left(\frac{\epsilon_2}{2}\right)^{n_1} \sqrt{\frac{[n_2]_{\mu_2}!\ \rho_{j}(\mu_2,\mu_1,n_1+n_2)}{[n_1]_{\mu_1}![n_1+n_2]_{\mu_2}!\ \nu_0(\mu_2,\mu_1,n_1+n_2)}}\ R_{n_1}(z_j;\mu_2,\mu_1,n_1+n_2)
    \end{align}
    with the $\mu$-factorial defined by $[n]_\mu! = [1]_\mu[2]_\mu \dots [n]_\mu$. Here, we fix the phase factors to be
    \begin{align}
    \phi(n_1,n_2,j) = \frac{n_1(n_1-1)}{2}+\frac{j(j+1)}{2}+n_1(n_1+n_2+1).
    \end{align}
    The Clebsch-Gordan decomposition can also be used to recouple the multiple tensor product of irreducible representations in a pairwise fashion. 
    For instance, when considering the three-fold tensor product $(\mu_{1},\epsilon_{1})\otimes(\mu_{2},\epsilon_{2})\otimes(\mu_{3},\epsilon_{3})$ there are two standard ways of doing this. On the one hand, one can decompose the first two spaces into irreducible representations, and then couple the resulting spaces with the third component :
    \begin{align} \label{decomp1}
    (\mu_{1},\epsilon_{1})\otimes(\mu_{2},\epsilon_{2})\otimes(\mu_{3},\epsilon_{3}) = \bigoplus_{j_{12}=0}^\infty (\mu_{12},\epsilon_{12})\otimes(\mu_{3},\epsilon_{3})
    = \bigoplus_{j_{12}=0}^\infty \bigoplus_{j_{(12)3}=0}^\infty (\mu_{123},\epsilon_{123}).
    \end{align}
    On the other hand, it is possible to combine the last two spaces first and to bring in the first component subsequently : 
    \begin{align} \label{decomp2}
      (\mu_{1},\epsilon_{1})\otimes(\mu_{2},\epsilon_{2})\otimes(\mu_{3},\epsilon_{3}) = \bigoplus_{j_{23}=0}^\infty (\mu_{1},\epsilon_{1})\otimes(\mu_{23},\epsilon_{23})
    = \bigoplus_{j_{23}=0}^\infty \bigoplus_{j_{1(23)}=0}^\infty (\mu_{123},\epsilon_{123}).
    \end{align}
    Focusing on the parameters $\mu$, the following relations stem from the repeated use of \eqref{decomp_param} : 
    \begin{align}
    \mu_{123} &= \mu_1+\mu_2+\mu_3+1+j_{123}, \\
	      &= \mu_{12}+\mu_3+j_{(12)3}, \\
	      &= \mu_1 + \mu_{23}+j_{1(23)}
    \end{align}
    and
    \begin{align}
    \mu_{12} = \mu_1 + \mu_2 + j_{12}, \qquad
    \mu_{23} = \mu_2 + \mu_3 + j_{23}.
    \end{align}
    Analogous relations can be found for the parameters $\epsilon$. These equations imply that the five decomposition integers $j$ are constrained :
    \begin{align} \label{jrelations}
    j_{123} = j_{1(23)}+j_{23} = j_{(12)3}+j_{12}.
    \end{align}
    While only three decomposition integers are independent, it will be convenient to use all five to simplify the notation especially when dealing with indices. 
    Now, to each of the two decomposition schemes one can associate a basis. To the scheme \eqref{decomp1} corresponds
    \begin{align} \label{Rbasis1}
    f_{n_{123}}^{j_{123},j_{12}} &= \sum_{n_{12}+n_3} C_{n_{12},n_{3}}^{n_{123},j_{(12)3}} e_{n_{12}}^{(\mu_{12},\epsilon_{12})}\otimes e_{n_{3}}^{(\mu_{3},\epsilon_{3})} \\
    &= \sum_{n_{12}+n_3} \sum_{n_1+n_2} C_{n_{12},n_{3}}^{n_{123},j_{(12)3}} C_{n_{1},n_{2}}^{n_{12},j_{12}}\ e_{n_{1}}^{(\mu_{1},\epsilon_{1})}\otimes e_{n_{2}}^{(\mu_{2},\epsilon_{2})}\otimes e_{n_{3}}^{(\mu_{3},\epsilon_{3})}
    \end{align}
    where the sums run over $n_{123}+j_{123} = n_{12} + n_3$, $n_{12}+j_{12}=n_1+n_2$ and to \eqref{decomp2}, the basis
    \begin{align} \label{Rbasis2}
    g_{n_{123}}^{j_{123},j_{23}} &= \sum_{n_{1}+n_{23}} C_{n_{1},n_{23}}^{n_{123},j_{1(23)}} e_{n_{1}}^{(\mu_{1},\epsilon_{1})}\otimes e_{n_{23}}^{(\mu_{23},\epsilon_{23})} \\
    &= \sum_{n_{1}+n_{23}} \sum_{n_2+n_3} C_{n_{1},n_{23}}^{n_{123},j_{1(23)}} C_{n_{2},n_{3}}^{n_{23},j_{23}}\ e_{n_{1}}^{(\mu_{1},\epsilon_{1})}\otimes e_{n_{2}}^{(\mu_{2},\epsilon_{2})}\otimes e_{n_{3}}^{(\mu_{3},\epsilon_{3})}
    \end{align}
    where $n_{123}+j_{123} = n_1 + n_{23}$ and $n_{23}+j_{23}=n_2+n_3$. 
    The connection coefficients for these two bases are called the Racah coefficients. Explicitly, 
    \begin{align} \label{RDecomp}
    f_{n_{123}}^{j_{123},j_{12}} = \sum_{j_{23}=0}^{j_{123}} \mathcal{R}_{j_{12},j_{23},j_{123}}^{\mu_1,\mu_2,\mu_3} g_{n_{123}}^{j_{123},j_{23}}.
    \end{align}
    It has been shown in \cite{Genest2014a} that the Racah coefficients for the $\mathfrak{osp}(1|2)$ Lie superalgebra can be expressed in terms of Bannai-Ito polynomials given in \eqref{hypergeo-expr}. One has
    \begin{align}
    \mathcal{R}_{j_{12},j_{23},j_{123}}^{\mu_1,\mu_2,\mu_3} = (-1)^\varphi \epsilon_3^{j_{12}} \sqrt{\frac{w_{j_{23}}}{h_{j_{12}}}}~ B_{j_{12}}\left(x_{j_{23}};\tfrac{\mu_2+\mu_3}{2},\tfrac{\mu_1+(-1)^{j_{123}}\mu_{123}}{2},\tfrac{\mu_3-\mu_2}{2},\tfrac{(-1)^{j_{123}}\mu_{123}-\mu_1}{2}\right)
    \end{align}
    where the $x_k, w_k, h_n$ are given in equations (\ref{gridi}--\ref{Ortho-Norm2}) where the parameters of the Bannai-Ito polynomials $\rho_i,r_i, i=1,2$ are assumed to be the same as in the polynomial $B_{j_{12}}(x;\rho_1,\rho_2,r_1,r_2)$ above. The choice of phase factor relevant for this paper is 
    \begin{align}
    \varphi = j_{123}~\frac{(j_{12}-1) j_{12}}{2}+(j_{123}+1) \left(j_{23}+\frac{(j_{12}+1)j_{12}}{2}\right). 
    \end{align}

\section{Convolution identities}

  The construction of convolution identities for orthogonal polynomials that we are proposing here will proceed along the  following lines : select an appropriate self-adjoint element $X$ from the Lie superalgebra, construct its generalized eigenvectors in a given representation and study their overlaps in the tensor product of representations. In order to obtain orthogonal polynomials, this Lie superalgebra element should act as a three-term recurrence operator in the chosen representation. Furthermore, the chosen element should generate a coideal subalgebra to ensure proper behavior under tensor product of representations. The self-adjoint element we consider here is   
  \begin{align} \label{X_c}
  X_c = J_+ + J_- + cR
  \end{align}
  in the $\mathfrak{osp}(1|2)$ Lie superalgebra depending a single parameter $c\in\mathbb{R}$. While this is not the most general self-adjoint element in $\mathfrak{osp}(1|2)$, the addition of a $J_0$ term would break the coideal property given in equation \eqref{coidealP}. Indeed, this property is non-trivial for algebras with a twisted coproduct such as \eqref{coproduct}. See \cite{Koelink1998} for example, where this construction is done for $\mathfrak{su}(1,1)$ and $U_q(\mathfrak{su}(1,1))$ where the coproduct is untwisted in the first case and twisted in the second. 

  With the element $X_c$ at hand, we shall study its discrete series representations and compute its generalized eigenvectors in the next subsection. In the subsequent ones, we will respectively consider the two-fold and three-fold tensor product of irreducible representations and derive two convolution identities.     

  \subsection{Action of $X_c$ in the positive discrete series representation}

    The operator $X_c$ has a tridiagonal structure on the representation space $(\mu,\epsilon)$ defined in \eqref{osp12Action}. One has 
    \begin{align}
    X_c e_{n}^{(\mu,\epsilon)} = [n+1]_\mu^{1/2}e_{n+1}^{(\mu,\epsilon)} + c\epsilon(-1)^ne_{n}^{(\mu,\epsilon)} + [n]_\mu^{1/2}e_{n-1}^{(\mu,\epsilon)}.
    \end{align}
    Let $v_{\lambda}^c$ denote the eigenvector with eigenvalue $\lambda$ of $X_c$. Then, there is an expansion of the form
    \begin{align}
    v_\lambda^c = \sum_{n=0}^\infty a_n e_n^{(\mu,\epsilon)}, \qquad a_n\in\mathbb{R}.
    \end{align}
    Acting on both sides of this equation with the operator $X_c$ gives the following 3-term recurrence relation on the expansion coefficients $a_n$ :
    \begin{align} \label{3termRRinderivation}
    \lambda a_n = [n+1]_\mu^{1/2} a_{n+1} + c\epsilon(-1)^n a_{n} + [n]_\mu^{1/2} a_{n-1}.
    \end{align}
    There is a solution of the form $a_n = P_n(\lambda)\cdot a_0$ where the $P_n(\lambda)$ are orthogonal polynomials satisfying the recurrence relation
    \begin{align} 
    \lambda P_n(\lambda) = [n+1]_\mu^{1/2} P_{n+1}(\lambda) + c\epsilon(-1)^n P_{n}(\lambda) + [n]_\mu^{1/2} P_{n-1}(\lambda).
    \end{align}
    Comparing this equation with the recurrence relation \eqref{OpeghRR}, one sees directly that the $P_n(\lambda)$ are Specialized Chihara polynomials with the following parameters :
    \begin{align}
    P_n(\lambda) = P_n(\lambda;\mu,c\epsilon).
    \end{align}
    Taking into account the normalization factor in the orthogonality relation \eqref{OpeghOR}, the polynomials $P_n(\lambda;\mu,c\epsilon)$ are thus {\em orthonormal} with respect to the weight function 
    \begin{align}
    W(\lambda,\mu,c\epsilon) = \frac{w(\lambda,\mu,c\epsilon)}{2\Gamma(\mu+\tfrac12)}
    \end{align}
    on the interval $F=(-\infty,|c|~)\cup(~|c|,\infty)$ where the $w(\lambda,\mu,c\epsilon)$ are given in \eqref{OpeghW}. If one asks that the eigenvectors $v_\lambda^c$ be orthonormal, it is easy to see that this implies $a_0=1$ and that the generalized eigenvectors of $X_c$ are
    \begin{align} \label{XcEV}
    v_\lambda^c = \sum_{n=0}^\infty P_n(\lambda;\mu,c\epsilon)~ e_n^{(\mu,\epsilon)}.
    \end{align}
    Note that the series in \eqref{XcEV} does not converge in the representation space and should be considered as a formal eigenvector.
    We reformulate this result in the following proposition.
    \begin{proposition} \label{prop1}
    The unitary operator
    \begin{align}
    \begin{aligned}
    \Lambda :\quad \ell^2(\mathbb{Z_+}) &\to L^2(F,W(\lambda,\mu,c\epsilon)) \\
		    e_n^{(\mu,\epsilon)} \ \ &\mapsto P_n(\lambda,\mu,c\epsilon)
    \end{aligned}
    \end{align}
    is an intertwiner of the operator $M_\lambda$ which denotes multiplication by $\lambda$ on $L^2(F,W(\lambda,\mu,c\epsilon))$ and the operator $X_c$ acting in $\ell^2(\mathbb{Z_+})$ :
    \begin{align}
    M_\lambda \Lambda = \Lambda X_c.
    \end{align}
    \end{proposition}
    \begin{proof}
    The unitarity of $\Lambda$ is checked from the fact that it maps an orthonormal basis onto another one. The intertwining relation derives directly from the above computation.
    \end{proof}

  \subsection{Action on the tensor product space}

    We now wish to study the action of the coproduct of $X_c$ on a tensor product of irreducible representations $(\mu_1,\epsilon_1)\otimes(\mu_2,\epsilon_2)$ and obtain its generalized eigenvectors. A straightforward computation yields
    \begin{align} \label{coidealP}
    \Delta(X_c) = X_c \otimes R + 1 \otimes X_0
    \end{align}
    where $X_0$ denotes the operator $X_c$ with its parameter set to zero. In view of how it acts on the first space in the tensor product, it is natural to study the action of $\Delta(X_c)$ on vectors of the following form :
    \begin{align}
    \Delta(X_c) v_{\lambda_1}^c\otimes e_{n_2}^{(\mu_2,\epsilon_2)}
      &= v_{\lambda_1}^c \otimes (\lambda_1 R + X_0) e_{n_2}^{(\mu_2,\epsilon_2)}\\
      &= v_{\lambda_1}^c \otimes X_{\lambda_1}e_{n_2}^{(\mu_2,\epsilon_2)}
    \end{align}
    where $v_{\lambda_1}^c$ is an eigenvector of $X_c$ given by \eqref{XcEV} and $X_{\lambda_1}$ is the operator \eqref{X_c} with parameter $\lambda_1$. It follows that the generalized eigenvectors of $\Delta(X_c)$ are 
    \begin{align}
    v_{\lambda_1,\lambda_2}^c = \sum_{n_1,n_2=0}^\infty P_{n_1}(\lambda_1;\mu_1,c\epsilon_1) P_{n_2}(\lambda_2;\mu_2,\lambda_1\epsilon_2)\ e_{n_1}^{(\mu_1,\epsilon_1)}\otimes e_{n_2}^{(\mu_2,\epsilon_2)}
    \end{align}
    with eigenvalues $\lambda_2$.
    This allows us to establish the following proposition.
    \begin{proposition} \label{prop2}
    The unitary operator
    \begin{align}
    \begin{aligned} \label{Upsilon}
    \Upsilon :\quad \ell^2(\mathbb{Z_+})\otimes\ell^2(\mathbb{Z_+}) &\to L^2(G,W(\lambda_1,\mu_1,c\epsilon_1)W(\lambda_2,\mu_2,\lambda_1\epsilon_2)) \\
		    e_{n_1}^{(\mu_1,\epsilon_1)}\otimes e_{n_2}^{(\mu_2,\epsilon_2)} \ \ &\mapsto P_{n_1}(\lambda_1;\mu_1,c\epsilon_1) P_{n_2}(\lambda_2;\mu_2,\lambda_1\epsilon_2)
    \end{aligned}
    \end{align}
    with 
    \[ G = \{(\lambda_1,\lambda_2)\in\mathbb{R} ~\big|~ |\lambda_2|>|\lambda_1|>|c|\} \] 
    is an intertwiner of the operator $M_{\lambda_2}$ denoting multiplication by $\lambda_2$ on $L^2(G,W(\lambda_1,\mu_1,c\epsilon_1)W(\lambda_2,\mu_2,\lambda_1\epsilon_2))$ and the operator $\Delta(X_c)$ acting in $\ell^2(\mathbb{Z_+})\otimes\ell^2(\mathbb{Z_+})$ :
    \begin{align} \label{Intertwine2}
    M_{\lambda_2} \Upsilon = \Upsilon \Delta(X_c).
    \end{align}
    \end{proposition}
    \begin{proof}
     The proof is similar to that of the previous proposition. Unitarity follows from the mapping of an orthonormal basis onto another one and the intertwining relation is a restatement of the eigenvectors computed above.
    \end{proof}

    Now, in view of the Clebsch-Gordan decomposition \eqref{osp12Decomp} of $(\mu_1,\epsilon_1)\otimes(\mu_2,\epsilon_2)$ into irreducible representations, there exists another orthonormal basis $e_{N}^{(\mu_{12},\epsilon_{12})}$ with $N=0,1,\dots$ and
    \begin{align} \label{mu12}
    \mu_{12} = \mu_1 + \mu_2 + j + \tfrac12, \quad \epsilon_{12} = \epsilon_1\epsilon_2(-1)^j, \quad j=0,1,\dots
    \end{align}
    where $j$ labels the irreducible subspaces of the form $(\mu_{12},\epsilon_{12})$ in the tensor product space. This basis is often referred to as the {\em coupled basis}, while the basis $e_{n_1}^{(\mu_1,\epsilon_1)}\otimes e_{n_2}^{(\mu_2,\epsilon_2)}$ is called the {\em uncoupled basis}. The operator $\Upsilon$ also maps the coupled basis to a set of orthonormal polynomials in $L^2(G,W(\lambda_1,\mu_1,c\epsilon_1)W(\lambda_2,\mu_2,\lambda_1\epsilon_2))$. 

    \begin{proposition} \label{prop3}
    In $L^2(G,W(\lambda_1,\mu_1,c\epsilon_1)W(\lambda_2,\mu_2,\lambda_1\epsilon_2))$, we have
    \begin{align} \label{UpsiloneN}
    \Upsilon e_{N}^{(\mu_{12},\epsilon_{12})}(\lambda_1,\lambda_2) = P_N(\lambda_2;\mu_{12},c\epsilon_{12})\Upsilon e_{0}^{(\mu_{12},\epsilon_{12})}(\lambda_1,\lambda_2),
    \end{align}
    \begin{align} \label{Upsilone0}
    \Upsilon e_{0}^{(\mu_{12},\epsilon_{12})}(\lambda_1,\lambda_2) = K_j(\lambda_2;\mu_2,\mu_1;c)\ J_j\left( \frac{\epsilon_2\lambda_1}{\lambda_2};\ 2\mu_2, 2\mu_1, -\frac{c\epsilon_1\epsilon_2}{\lambda_2} \right)
    \end{align}
    with
    \begin{align} \label{KN}
    K_j(\lambda_2;\mu_2,\mu_1;c) = \sqrt{ \left(\frac{\lambda_2^2-c^2}{2}\right)^j \left(\frac{\lambda_2 - c\epsilon_1\epsilon_2}{\lambda_2 - (-1)^j c\epsilon_1\epsilon_2}\right) \frac{\Gamma(\mu_1+\tfrac12)\Gamma(\mu_2+\tfrac12)}{\Gamma(\mu_{12}+\tfrac12) h_j(2\mu_2,2\mu_1)} }
    \end{align}
    where the notation of section 2 for the Specialized Chihara and the Big \m1 Jacobi polynomials is assumed.
    \end{proposition}

    \begin{proof}
    Use the intertwining relation \eqref{Intertwine2} to write
    \begin{align}
    \lambda_2 \Upsilon e_{N}^{(\mu_{12},\epsilon_{12})} = M_{\lambda_2} \Upsilon e_{N}^{(\mu_{12},\epsilon_{12})} = \Upsilon \Delta(X_c) e_{N}^{(\mu_{12},\epsilon_{12})}.
    \end{align}
    The action of $\Delta(X_c)$ on the irreducible spaces $(\mu_{12},\epsilon_{12})$ is given by the relations \eqref{osp12Action}. This yields a 3-term recurrence relation on $N$ of the same form as \eqref{3termRRinderivation}. Its solution is \eqref{UpsiloneN} where the initial condition $\Upsilon e_{0}^{(\mu_{12},\epsilon_{12})}(\lambda_1,\lambda_2)$ remains to be determined. To obtain it, note that the orthonormality of the vectors and the unitarity of $\Upsilon$ implies the following. Using the notation \eqref{mu12} and $\tilde{\mu}_{12} = \mu_1 + \mu_2 + \tfrac12 + \tilde{j},\ \tilde{\epsilon}_{12} = \epsilon_1\epsilon_2(-1)^{\tilde{j}}$, one has 
    \begin{align}
    \begin{aligned}
    \delta_{N,\tilde{N}} \delta_{j,\tilde{j}} &= \langle e_N^{(\mu_{12},\epsilon_{12})} , e_{\tilde{N}}^{(\tilde{\mu}_{12},\tilde{\epsilon}_{12})} \rangle 
    = \langle \Upsilon e_N^{(\mu_{12},\epsilon_{12})} , \Upsilon e_{\tilde{N}}^{(\tilde{\mu}_{12},\tilde{\epsilon}_{12})} \rangle \\[0.5em]
    &= \iint_G \ P_N(\lambda_2;\mu_{12},c\epsilon_{12})P_{\tilde{N}}(\lambda_2;\tilde{\mu}_{12},c\tilde{\epsilon}_{12})\Upsilon e_{0}^{(\mu_{12},\epsilon_{12})}(\lambda_1,\lambda_2) \Upsilon e_{0}^{(\tilde{\mu}_{12},\tilde{\epsilon}_{12})}(\lambda_1,\lambda_2) \\[0.5em] 
    &\qquad\times W(\lambda_1,\mu_1,c\epsilon_1)W(\lambda_2,\mu_2,\lambda_1\epsilon_2)d\lambda_1 d\lambda_2.
    \end{aligned}
    \end{align}
    Integrate first on $\lambda_1$ :
    \begin{align}
    \begin{aligned}
    \delta_{N,\tilde{N}} \delta_{j,\tilde{j}} &= 
    \int_{G_2} P_N(\lambda_2;\mu_{12},c\epsilon_{12})P_{\tilde{N}}(\lambda_2;\tilde{\mu}_{12},c\tilde{\epsilon}_{12}) \\[0.5em] 
    &\times \int_{G_1} \Upsilon e_{0}^{(\mu_{12},\epsilon_{12})}(\lambda_1,\lambda_2) \Upsilon e_{0}^{(\tilde{\mu}_{12},\tilde{\epsilon}_{12})}(\lambda_1,\lambda_2) 
      W(\lambda_1,\mu_1,c\epsilon_1)W(\lambda_2,\mu_2,\lambda_1\epsilon_2)d\lambda_1 d\lambda_2
    \end{aligned}
    \end{align}
    where 
    \begin{align}
    G_1 = (-\lambda_2,-|c|)\cup(|c|,\lambda_2), \qquad G_2 = (-\infty,-|c|)\cup(|c|,\infty).
    \end{align}
    If we consider the special case $j=\tilde{j}$, the inner integral on $\lambda_1$ must correspond to the orthogonality measure of the Specialized Chihara polynomials $P_N(\lambda_2;\mu_{12},c\epsilon_{12})$ since the corresponding moment problem is determined. This can be checked from the divergence of the series $\sum_{n=1}^\infty [n]_\mu^{-1/2}$ which satisfies one of Carleman's conditions for determinacy \cite{Koelink2004}. Furthermore, applying $\Upsilon$ on both sides of the Clebsch-Gordan decomposition \eqref{CGDecomp} with $N=0$, one can deduce that $\Upsilon e_0^{(\mu_{12},\epsilon_{12})}$ is a polynomial of degree $j$ in the variables $\lambda_1$ and $\lambda_2$. Taking these two observations into account, it is possible to identify $\Upsilon e_{0}^{(\mu_{12},\epsilon_{12})}$ in terms of Big $-1$ Jacobi polynomials. Indeed, setting $u = \epsilon_2\lambda_1/\lambda_2$, it is straightforward to identify the resulting weights factors with those of the Big $-1$ Jacobi polynomials given in \eqref{Weight-1} in the variable $u$. This gives the result \eqref{Upsilone0}. The normalization factor \eqref{KN} is obtained by comparing the weight function in the integral on $\lambda_2$ with that of the Specialized Chihara polynomials $P_N(\lambda_2;\mu_{12},c\epsilon_{12})$ given in \eqref{OpeghW} and computing the integral.
    \end{proof}

    Thus, the equations \eqref{UpsiloneN} and \eqref{Upsilone0} give us the action of $\Upsilon$ on the coupled basis. With these results in hand, it is now possible to obtain a convolution identity for the Specialized Chihara, Big \m1 Jacobi and dual \m1 Hahn polynomials. 

    \begin{proposition} \label{convid1}
    In the notation of section 2 for the Specialized Chihara, Big \m1 Jacobi and dual \m1 Hahn polynomials, the following convolution identity holds :
    \begin{align} \label{conv_id_CG}
    \begin{aligned}
      K_j(\lambda_2;\mu_2,\mu_1;c)\ &P_N(\lambda_2;\mu_{12},c\epsilon_{12}) J_j\left( \frac{\epsilon_2\lambda_1}{\lambda_2};\ 2\mu_2, 2\mu_1, -\frac{c\epsilon_1\epsilon_2}{\lambda_2} \right) \\[0.5em] 
      &= \sum_{n_1+n_2=N+j} (-1)^{\phi(n_1,n_2,j)} \left(\frac{\epsilon_2}{2}\right)^{n_1} \sqrt{\frac{[n_2]_{\mu_2}!\ \rho_{j}(\mu_2,\mu_1,n_1+n_2)}{[n_1]_{\mu_1}![n_1+n_2]_{\mu_2}!\ \nu_0(\mu_2,\mu_1,n_1+n_2)}}\ \\[0.5em] 
      &\hspace{2.6cm} \times R_{n_1}(z_j;\mu_2,\mu_1,n_1+n_2)P_{n_1}(\lambda_1;\mu_1,c\epsilon_1) P_{n_2}(\lambda_2;\mu_2,\lambda_1\epsilon_2).
    \end{aligned}
    \end{align}
    \end{proposition}
    \begin{proof}
    Consider the Clebsch-Gordan decomposition \eqref{CGDecomp} and apply the operator $\Upsilon$ on each side of the equation. Using \eqref{Upsilon}, \eqref{UpsiloneN} and \eqref{Upsilone0}, one obtains the above formula. 
    \end{proof}
    \begin{Remark}
     It is also possible to obtain another convolution identity using the inverse basis expansion :
    \begin{align} \label{conv_id_CG2}
      P_{n_1}(\lambda_1;\mu_1,c\epsilon_1) &P_{n_2}(\lambda_2;\mu_2,\lambda_1\epsilon_2) \notag\\ 
      &= \sum_{N+j=n_1+n_2} (-1)^{\phi(n_1,n_2,j)} \left(\frac{\epsilon_2}{2}\right)^{n_1} \sqrt{\frac{[n_2]_{\mu_2}!\ \rho_{j}(\mu_2,\mu_1,n_1+n_2)}{[n_1]_{\mu_1}![n_1+n_2]_{\mu_2}!\ \nu_0(\mu_2,\mu_1,n_1+n_2)}}\ \\[0.5em] \notag
      & \times R_{n_1}(z_j;\mu_2,\mu_1,n_1+n_2)K_j(\lambda_2;\mu_2,\mu_1;c)\ P_N(\lambda_2;\mu_{12},c\epsilon_{12}) J_j\left( \frac{\epsilon_2\lambda_1}{\lambda_2};\ 2\mu_2, 2\mu_1, -\frac{c\epsilon_1\epsilon_2}{\lambda_2} \right).
    \end{align}
    \end{Remark}

    \begin{Remark}
     Note that the propositions \ref{prop2} and \ref{prop3} both define polynomials in two variables $\lambda_1, \lambda_2$ of the Tratnik-type that are orthogonal with respect to the same measure $W(\lambda_1,\mu_1,c\epsilon_1)W(\lambda_2,\mu_2,\lambda_1\epsilon_2)$ on $G$. In this picture, the convolution identities \eqref{conv_id_CG} and \eqref{conv_id_CG2} provide connection coefficients for these two sets of orthogonal polynomials. This is an interesting result especially since the generalization to multiple variables of the Bannai-Ito scheme is still in its early stages. Note that section 4.3 similarly gives orthogonal polynomials in three variables in this class.  
    \end{Remark}

  \subsection{Action on the three-fold tensor product space}

    We now study how the operator $X_c$ can be extended to the three-fold tensor product space $(\mu_1,\epsilon_1)\otimes(\mu_2,\epsilon_2)\otimes(\mu_3,\epsilon_3)$. This will lead to another convolution identity involving the Racah coefficients of $\mathfrak{osp}(1|2)$. One first computes 
    \begin{align}
    \begin{aligned}
    \Delta^2(X_c)\equiv (1\otimes \Delta)\Delta(X_c) &= \Delta(X_c)\otimes R + \Delta(1) \otimes X_0 \\
      &= X_c \otimes R \otimes R + 1 \otimes X_0 \otimes R + 1 \otimes 1 \otimes X_0.
    \end{aligned}
    \end{align}
    The notation $\Delta^2$ is unambiguous because of the coassociativity of the coproduct : $(1\otimes \Delta)\Delta= (\Delta\otimes 1)\Delta$.
    The eigenvectors of $\Delta^2(X_c)$ can be found by studying its action on vectors of the form $v_{\lambda_1,\lambda_2}^c \otimes e_{n_3}$ :
    \begin{align}
    \begin{aligned}
    \Delta^2(X_c) v_{\lambda_1,\lambda_2}^c \otimes e_{n_3} &= (\Delta(X_c)\otimes R + \Delta(1) \otimes X_0)v_{\lambda_1,\lambda_2}^c \otimes e_{n_3} \\
      &= (1\otimes1\otimes \lambda_2 R + 1\otimes1\otimes X_0)v_{\lambda_1,\lambda_2}^c \otimes e_{n_3} \\
      &= (1\otimes1\otimes X_{\lambda_2})v_{\lambda_1,\lambda_2}^c \otimes e_{n_3} \\
      &= v_{\lambda_1,\lambda_2}^c \otimes X_{\lambda_2}e_{n_3}.
    \end{aligned}
    \end{align}
    It follows naturally that the generalized eigenvectors are
    \begin{align}
    v_{\lambda_1,\lambda_2,\lambda_3}^c = \sum_{n_1,n_2,n_3=0}^\infty P_{n_1}(\lambda_1;\mu_1,c\epsilon_1) P_{n_2}(\lambda_2;\mu_2,\lambda_1\epsilon_2)P_{n_3}(\lambda_3;\mu_3,\lambda_2\epsilon_3)\ e_{n_1}^{(\mu_1,\epsilon_1)}\otimes e_{n_2}^{(\mu_2,\epsilon_2)}\otimes e_{n_3}^{(\mu_3,\epsilon_3)}
    \end{align}
    with eigenvalues $\lambda_3$. This establishes the analog of propositions \ref{prop1} and \ref{prop2}.
    \begin{proposition} \label{prop4}
    The unitary operator
    \begin{align} \label{Theta}
    \Theta :\quad \ell^2(\mathbb{Z_+})\otimes\ell^2(\mathbb{Z_+})\otimes\ell^2(\mathbb{Z_+}) &\to L^2(G^{(3)},W^{(3)}), \\
		    e_{n_1}^{(\mu_1,\epsilon_1)}\otimes e_{n_2}^{(\mu_2,\epsilon_2)}\otimes e_{n_3}^{(\mu_3,\epsilon_3)} \ \ &\mapsto P_{n_1}(\lambda_1;\mu_1,c\epsilon_1) P_{n_2}(\lambda_2;\mu_2,\lambda_1\epsilon_2)P_{n_3}(\lambda_3;\mu_3,\lambda_2\epsilon_3)
    \end{align}
    where 
    \[ G^{(3)} = \{(\lambda_1,\lambda_2,\lambda_3)\in\mathbb{R} ~\big|~ |\lambda_3|>|\lambda_2|>|\lambda_1|>|c|\} \] 
    and
    \[  W^{(3)} = W(\lambda_1,\mu_1,c\epsilon_1)W(\lambda_2,\mu_2,\lambda_1\epsilon_2)W(\lambda_3,\mu_3,\lambda_2\epsilon_3).  \]
    is an intertwiner of the operator $M_{\lambda_3}$ on $L^2(G^{(3)},W^{(3)})$ and the operator $\Delta^2(X_c)$ acting in $\ell^2(\mathbb{Z_+})\otimes\ell^2(\mathbb{Z_+})\otimes\ell^2(\mathbb{Z_+})$ :
    \begin{align} \label{Intertwine3}
    M_{\lambda_3} \Theta = \Theta \Delta^2(X_c).
    \end{align}
    \end{proposition}
    \begin{proof}
    The proof follows those of propositions \ref{prop1} and \ref{prop2}. Unitarity follows from mapping an orthonormal basis onto another one. The intertwining relation comes from the computation of the generalized eigenvectors before proposition \ref{prop4}.
    \end{proof}

    The idea is again to act with the operator $\Theta$ in different bases in order to obtain a new convolution identity. The bases of interest here are the two that arise when decomposing the representation space $(\mu_{1},\epsilon_{1})\otimes(\mu_{2},\epsilon_{2})\otimes(\mu_{3},\epsilon_{3})$ into irreducible components. As mentioned in section 3, this can be done in two standard ways using the Clebsch-Gordan decomposition. This yields the bases $f_{n_{123}}^{j_{123},j_{12}}$ and $g_{n_{123}}^{j_{123},j_{23}}$ respectively given in \eqref{Rbasis1} and \eqref{Rbasis2}.

    \begin{proposition} \label{prop5}
    In $L^2(G^{(3)},W^{(3)})$, we have 
    \begin{align}
    \Theta f_{n_{123}}^{j_{123},j_{12}} =~  &K_{j_{12}}(\lambda_2;\mu_2,\mu_1;c)~    J_{j_{12}}\left(\frac{\epsilon_2\lambda_1}{\lambda_2};~2\mu_2, 2\mu_1,-\frac{c\epsilon_1\epsilon_2}{\lambda_2} \right) \\
				    \times &K_{j_{(12)3}}(\lambda_3;\mu_3,\mu_{12};c)~J_{j_{(12)3}}\left(\frac{\epsilon_3\lambda_2}{\lambda_3};~2\mu_3, 2\mu_{12},-\frac{c\epsilon_{12}\epsilon_3}{\lambda_3} \right) 
				    P_{n_{123}}(\lambda_3; \mu_{123},c\epsilon_{123}),\notag \\[1em]
    \Theta g_{n_{123}}^{j_{123},j_{23}} =~  &K_{j_{23}}(\lambda_3;\mu_3,\mu_2;\lambda_1)~    J_{j_{23}}\left(\frac{\epsilon_3\lambda_2}{\lambda_3};~2\mu_3, 2\mu_2,-\frac{\lambda_1\epsilon_2\epsilon_3}{\lambda_3} \right) \\
				    \times &K_{j_{1(23)}}(\lambda_3;\mu_{23},\mu_{1};c)~J_{j_{1(23)}}\left(\frac{\epsilon_{23}\lambda_1}{\lambda_3};~2\mu_{23}, 2\mu_{1},-\frac{c\epsilon_{1}\epsilon_{23}}{\lambda_3} \right) 
				    P_{n_{123}}(\lambda_3; \mu_{123},c\epsilon_{123}).\notag
    \end{align}
    \end{proposition}
    \begin{proof}
    The key to obtain the action of $\Theta$ on these two bases is to use the expansions \eqref{Rbasis1} and \eqref{Rbasis2} in terms of the basis $e_{n_{1}}^{(\mu_{1},\epsilon_{1})}\otimes e_{n_{2}}^{(\mu_{2},\epsilon_{2})}\otimes e_{n_{3}}^{(\mu_{3},\epsilon_{3})}$, act with $\Theta$ as per proposition \ref{prop4} and resum the resulting polynomials by using the convolution identity \eqref{conv_id_CG} twice. Using the notation of proposition \ref{prop3} and of section 2, one obtains the relations above.
    \end{proof}

    This leads to the following result.

    \begin{proposition}\label{convid2} 
    The convolution identity
    \begin{align}
    &\notag K_{j_{12}}(\lambda_2;\mu_2,\mu_1;c)  J_{j_{12}}\left(\tfrac{\epsilon_2\lambda_1}{\lambda_2};~2\mu_2, 2\mu_1,-\tfrac{c\epsilon_1\epsilon_2}{\lambda_2} \right)     K_{j_{(12)3}}(\lambda_3;\mu_3,\mu_{12};c)J_{j_{(12)3}}\left(\tfrac{\epsilon_3\lambda_2}{\lambda_3};~2\mu_3, 2\mu_{12},-\tfrac{c\epsilon_{12}\epsilon_3}{\lambda_3} \right)\\
    &= \sum_{j_{23}=0}^{j_{123}} (-1)^\varphi \epsilon_3^{j_{12}} \sqrt{\frac{w_{j_{23}}}{h_{j_{12}}}}~ B_{j_{12}}\left(x_{j_{23}};\tfrac{\mu_2+\mu_3}{2},\tfrac{\mu_1+(-1)^{j_{123}}\mu_{123}}{2},\tfrac{\mu_3-\mu_2}{2},\tfrac{(-1)^{j_{123}}\mu_{123}-\mu_1}{2}\right) \\
    &\notag \times K_{j_{23}}(\lambda_3;\mu_3,\mu_2;\lambda_1)  J_{j_{23}}\left(\tfrac{\epsilon_3\lambda_2}{\lambda_3};~2\mu_3, 2\mu_2,-\tfrac{\lambda_1\epsilon_2\epsilon_3}{\lambda_3} \right)  K_{j_{1(23)}}(\lambda_3;\mu_{23},\mu_{1};c) J_{j_{1(23)}}\left(\tfrac{\epsilon_{23}\lambda_1}{\lambda_3};~2\mu_{23}, 2\mu_{1},-\tfrac{c\epsilon_{1}\epsilon_{23}}{\lambda_3} \right)
    \end{align}
    holds with the relations \eqref{jrelations} between the $j_{12}, j_{23}, j_{(12)3}, j_{1(23)}, j_{123}$, the notation for the polynomials of section 2 and equation \eqref{KN}.
    \end{proposition}
    \begin{proof}
    This formula is obtained by acting with $\Theta$ on both sides of the Racah decomposition \eqref{RDecomp}. The factors $P_{n_{123}}(\lambda_3; \mu_{123},c\epsilon_{123})$ on the left and on the right cancel out. This is just the manifestation of the well-known Wigner-Eckart theorem in this context.
    \end{proof}

    \begin{Remark}
    It is also possible to obtain a similar convolution identity using the orthogonality of the Racah coefficients :
    \begin{align}
    &\notag K_{j_{23}}(\lambda_3;\mu_3,\mu_2;\lambda_1)   J_{j_{23}}\left(\tfrac{\epsilon_3\lambda_2}{\lambda_3};~2\mu_3, 2\mu_2,-\tfrac{\lambda_1\epsilon_2\epsilon_3}{\lambda_3} \right)   K_{j_{1(23)}}(\lambda_3;\mu_{23},\mu_{1};c) J_{j_{1(23)}}\left(\tfrac{\epsilon_{23}\lambda_1}{\lambda_3};~2\mu_{23}, 2\mu_{1},-\tfrac{c\epsilon_{1}\epsilon_{23}}{\lambda_3} \right)\\
    &= \sum_{j_{12}=0}^{j_{123}} (-1)^\varphi \epsilon_3^{j_{12}} \sqrt{\frac{w_{j_{23}}}{h_{j_{12}}}}~ B_{j_{12}}\left(x_{j_{23}};\tfrac{\mu_2+\mu_3}{2},\tfrac{\mu_1+(-1)^{j_{123}}\mu_{123}}{2},\tfrac{\mu_3-\mu_2}{2},\tfrac{(-1)^{j_{123}}\mu_{123}-\mu_1}{2}\right) \\
    &\notag \times K_{j_{12}}(\lambda_2;\mu_2,\mu_1;c)    J_{j_{12}}\left(\tfrac{\epsilon_2\lambda_1}{\lambda_2};~2\mu_2, 2\mu_1,-\tfrac{c\epsilon_1\epsilon_2}{\lambda_2} \right)     K_{j_{(12)3}}(\lambda_3;\mu_3,\mu_{12};c) J_{j_{(12)3}}\left(\tfrac{\epsilon_3\lambda_2}{\lambda_3};~2\mu_3, 2\mu_{12},-\tfrac{c\epsilon_{12}\epsilon_3}{\lambda_3} \right)
    \end{align}
    where the relations \eqref{jrelations} and the same notation are still assumed.
    \end{Remark}

\section{Bilinear generating function}

In this section, we consider a realization of $\mathfrak{osp}(1|2)$ in terms of Dunkl operators. This leads to a generating function for the Specialized Chihara polynomials. Additionally, the convolution identity from proposition \ref{convid1} is used to derive a bilinear generating function for the Big \m1 Jacobi polynomials.   

We introduce the following realization in terms of Dunkl operators of the $\mathfrak{osp}(1|2)$ Lie superalgebra :
\begin{align} \label{realn}
 J_+ = z, \quad
 J_- = \partial_z + \frac{\mu}{z}(1-R_z),\quad
 J_0 = z\partial_z + \mu + \tfrac12,\quad
 R = R_z
\end{align}
where $R_z$ is just the the reflexion operator acting on the variable $z$ by $R_z f(z)=f(-z)$. These operators verify the relations \eqref{osp12Action} when acting on the orthonormal basis vectors $e_n^{(\mu,\epsilon)} = ([n]_\mu!)^{-1/2}~ z^n$. 

We shall make use of this model to derive generating functions for the Specialized Chihara polynomials $P_n(\lambda;\mu,c\epsilon)$. First, recall equation \eqref{XcEV} giving the generalized eigenvectors of the operator $X_c$ and insert the realization \eqref{realn} above to obtain 
\begin{align} \label{5p2}
 v_\lambda^c(z,\mu,\epsilon) = \sum_{n=0}^\infty P_n(\lambda;\mu,c\epsilon)~ \frac{z^n}{[n]_\mu^{1/2}!}.
\end{align}
If one can obtain an explicit formula for $v_\lambda^c(z,\mu,\epsilon)$ in terms of special functions, this equation will yield the desired generating function. This can be done through the following steps. First, split the sum in the RHS over the even and the odd values of $n$ :
\begin{align} \label{splitpar}
 v_\lambda^c(z,\mu,\epsilon) = \sum_{k=0}^\infty P_{2k}(\lambda;\mu,c\epsilon)~ \frac{z^{2k}}{[2k]_\mu^{1/2}!} + \sum_{k=0}^\infty P_{2k+1}(\lambda;\mu,c\epsilon)~ \frac{z^{2k+1}}{[2k+1]_\mu^{1/2}!}.
\end{align}
Then, rewrite the $\mu$-factorial in terms of Pochhammer symbols with 
\begin{align}
 [2k]_\mu! = 4^k k!(\mu+\tfrac12)_k, \qquad [2k+1]_\mu! = 2(\mu+\tfrac12) 4^k k!(\mu+\tfrac32)_k 
\end{align}
and express the polynomials $P_{2k}(\lambda)$ and $P_{2k+1}(\lambda)$ in terms of Laguerre polynomials using equation \eqref{ExplicitOpegh} to get
\begin{align}
 v_\lambda^c(z,\mu,\epsilon) = \sum_{k=0}^\infty \frac{(-\tfrac12 z^2)^k}{(\mu+\tfrac12)_k} L_k^{(\mu-\frac12)}\left(\frac{\lambda^2-c^2}{2}\right)~
 + ~\frac{z(\lambda - c\epsilon)}{2\mu+1}~ \sum_{k=0}^\infty \frac{(-\tfrac12 z^2)^k}{(\mu+\tfrac32)_k} L_k^{(\mu+\frac12)}\left(\frac{\lambda^2-c^2}{2}\right).
\end{align}
Now, each sum can be reframed in terms of hypergeometric functions with the help of a generating function for the Laguerre polynomials (see equation 1.11.11 of \cite{Koekoek2010}) :
\begin{align}
 \sum_{n=0}^\infty \frac{t^n}{(\alpha+1)_n} L_n^{(\alpha)}(x) = e^t ~ \pFq{0}{1}{-}{\alpha+1}{-xt}.
\end{align}
Inserting the result into equation \eqref{5p2} leads to the following generating function for the Specialized Chihara polynomials.
\begin{proposition} \label{prop6}
The Specialized Chihara polynomials possess the generating function
\begin{align} 
\sum_{n=0}^\infty P_n(\lambda;\mu,c\epsilon)~ \frac{z^n}{[n]_\mu^{1/2}!}
 = e^{-z^2/2} \left( \pFq{0}{1}{-}{\mu+\tfrac12}{\frac{z^2 (\lambda^2 - c^2)}{4}}  +   \frac{z(\lambda-c\epsilon)}{2\mu+1} \pFq{0}{1}{-}{\mu+\tfrac32}{\frac{z^2 (\lambda^2 - c^2)}{4}} \right). 
\end{align}
\end{proposition}
\begin{proof}
 This result follows from the preceding computation.
\end{proof}

\begin{Remark}
It is also possible to reexpress the hypergeometric functions in terms of the modified Bessel functions of the first type $I_\alpha(x)$ to obtain 
\begin{align} 
\sum_{n=0}^\infty P_n(\lambda;\mu,c\epsilon)~ \frac{z^n}{[n]_\mu^{1/2}!}
 = \frac{e^{-z^2/2} \Gamma(\mu-\tfrac12)}{(\tfrac{z}{2})^{\mu-1/2}(\lambda^2 - c^2)^{\frac{\mu}{2}+\frac14}} \left[ (\lambda^2 - c^2)^{\frac14} I_{\mu-\frac12}(z{\scriptstyle \sqrt{\lambda^2-c^2}})  +  \tfrac{2\mu-1}{2\mu+1}~ I_{\mu+\frac12}(z{\scriptstyle \sqrt{\lambda^2-c^2}}) \right].
\end{align}
\end{Remark}
\begin{Remark}
Note that this explicit expression for the generalized eigenvectors $v_\lambda^c(z,\mu,\epsilon)$ in terms of special functions in the realization \eqref{realn} can also be obtained by solving the difference-differential equation $X_c v_\lambda^c(z,\mu,\epsilon) = \lambda v_\lambda^c(z,\mu,\epsilon)$. This also requires to separate the function $v_\lambda^c(z,\mu,\epsilon)$ into an even and odd part and leads to a system of two coupled first order ordinary differential equations. 
\end{Remark}

Focusing now on obtaining a bilinear generating function, we consider how the generating function given above carries to the tensor product of representations. In fact, the extension of the generalized eigenvectors to the tensor product of two representations in the realization \eqref{realn} is immediate. Explicitly, the generalized eigenvectors of $\Delta^2(X_c)$ are
\begin{align}
 v_{\lambda_1}^c(z_1,\mu_1,\epsilon_1)v_{\lambda_2}^{\lambda_1}(z_2,\mu_2,\epsilon_2) = \sum_{n_1, n_2=0}^\infty P_{n_1}(\lambda_1;\mu_1,c\epsilon_1) P_{n_2}(\lambda_2;\mu_2,\lambda_1\epsilon_2) \frac{z_1^{n_1}z_2^{n_2}}{[n_1]_{\mu_1}^{1/2}! [n_2]_{\mu_2}^{1/2}!}
\end{align}
where each $v_\lambda^c(z,\mu,\epsilon)$ admits an expression in terms of special functions as before. The monomials can be cast in the coupled basis by the inverse expansion of \eqref{CGDecomp}, 
\begin{align}
 \frac{z_1^{n_1}z_2^{n_2}}{[n_1]_{\mu_1}^{1/2}! [n_2]_{\mu_2}^{1/2}!} = \sum_{N+j=n_1+n_2} C_{n_1,n_2}^{N,j} e_N^{(\mu_{12},\epsilon_{12})}(z_1,z_2),
\end{align}
to obtain 
\begin{align*}
 v_{\lambda_1}^c(z_1,\mu_1,\epsilon_1)v_{\lambda_2}^{\lambda_1}(z_2,\mu_2,\epsilon_2) = \sum_{n_1, n_2=0}^\infty P_{n_1}(\lambda_1;\mu_1,c\epsilon_1) P_{n_2}(\lambda_2;\mu_2,\lambda_1\epsilon_2) \sum_{N+j=n_1+n_2} C_{n_1,n_2}^{N,j} e_N^{(\mu_{12},\epsilon_{12})}(z_1,z_2).
\end{align*}
Inverting the order of summation, one gets
\begin{align*}
 v_{\lambda_1}^c(z_1,\mu_1,\epsilon_1)v_{\lambda_2}^{\lambda_1}(z_2,\mu_2,\epsilon_2) = \sum_{N, j=0}^\infty e_N^{(\mu_{12},\epsilon_{12})}(z_1,z_2) \sum_{n_1+n_2=N+j} C_{n_1,n_2}^{N,j} P_{n_1}(\lambda_1;\mu_1,c\epsilon_1) P_{n_2}(\lambda_2;\mu_2,\lambda_1\epsilon_2)
\end{align*}
where the second sum now corresponds directly to the first convolution identity \eqref{conv_id_CG}. Using this gives 
\begin{align} \label{gendev}
\begin{aligned}
 v_{\lambda_1}^c(z_1,\mu_1,\epsilon_1)v_{\lambda_2}^{\lambda_1}(z_2,\mu_2,\epsilon_2) = &\sum_{j=0}^\infty K_j(\lambda_2;\mu_2,\mu_1;c) J_j\left( \tfrac{\epsilon_2\lambda_1}{\lambda_2};\ 2\mu_2, 2\mu_1, -\tfrac{c\epsilon_1\epsilon_2}{\lambda_2} \right) \\ &\times \sum_{N=0}^\infty P_N(\lambda_2;\mu_{12},c\epsilon_{12}) e_N^{(\mu_{12},\epsilon_{12})}(z_1,z_2).
\end{aligned}
\end{align}
If one obtains an expression for the second sum in terms of special functions, then a generating function for the Big \m1 Jacobi polynomials follows immediately. We first look for an explicit realization of the coupled basis vectors $e_N^{(\mu_{12},\epsilon_{12})}(z_1,z_2)$. Consider its expansion in terms of the uncoupled basis given in \eqref{CGDecomp} in the realization \eqref{realn} and substitute $n_2= N+j-n_1$ :   
\begin{align}
 e_N^{(\mu_{12},\epsilon_{12})}(z_1,z_2) = \sum_{n_1=0}^{N+j} C_{n_1,N+j-n_1}^{N,j} \frac{(\frac{z_1}{z_2})^{n_1}}{[n_1]_{\mu_1}^{1/2}! [N+j-n_1]_{\mu_2}^{1/2}!} ~z_2^{N+j}.
\end{align}
The sum on the RHS corresponds to the generating function for the $\mathfrak{osp}(1|2)$ Clebsch-Gordan coefficients \cite{Bergeron2016}. Taking into account the choice of normalization and phase factor made in section 3, this gives 
\begin{align} \label{eNreal}
 e_N^{(\mu_{12},\epsilon_{12})}(z_1,z_2) = 
\begin{cases}
  \dfrac{\left(z_1^2+z_2^2\right)^{N/2}}{[N]_{\mu_{12}}^{1/2}!} f_e(j) \qquad \text{if $N$ even,} \\[1.5em]
  \dfrac{\left(z_1^2+z_2^2\right)^{N/2}}{[N]_{\mu_{12}}^{1/2}!} f_o(j) \qquad \text{if $N$ odd}
\end{cases}
\end{align}
where $f_e(j)$ and $f_o(j)$ are functions of $j$. Let $j=2j_e+j_p$ with $j_p \in \{0,1\}$ and $j_e \in \mathbb{N}$, then  
\begin{align} \label{fe}
 f_e&(j) = (-1)^{j_e+j_p}\frac{z_2^j}{[j]_{\mu_2}^{1/2}!} \left[ \frac{(\frac12 +\mu_1)_{j_e+j_p}}{(j_e+j_p+1+\mu_1+\mu_2)_{j_e+j_p}} \right]^{1/2} \\[0.5em]
 \notag    &\times \left( \pFq{2}{1}{-j_e, \tfrac12-j_e-j_p-\mu_2}{\tfrac12+\mu_1}{-(\tfrac{z_1}{z_2})^2} + \frac{(-1)^{j_p}z_1(j+2\mu_2 j_p)}{z_2\epsilon_2(1+2\mu_1)} \pFq{2}{1}{1-j_e-j_p, \tfrac12-j_e-\mu_2}{\tfrac32+\mu_1}{-(\tfrac{z_1}{z_2})^2} \right)
\end{align}
and
\begin{align} \label{fo}
 f_o&(j) = (-1)^{j_e+j_p}\left(\frac{z_1^2}{z_2^2}+1\right)^{-1/2} \frac{z_2^j}{[j]_{\mu_2}^{1/2}!} \left[ \frac{(\frac12 +\mu_1)_{j_e+j_p}}{(j_e+j_p+1+\mu_1+\mu_2)_{j_e+j_p}} \right]^{1/2} \\[0.5em]
 \notag    &\times \left( \pFq{2}{1}{-j_e-j_p, -\tfrac12-j_e-\mu_2}{\tfrac12+\mu_1}{-(\tfrac{z_1}{z_2})^2} + \tfrac{(-1)^{j_p} z_1(j+1+2\mu_1+2\mu_2 j_p)}{z_2\epsilon_2(1+2\mu_1)} \pFq{2}{1}{-j_e, \tfrac12\!-\!j_e\!-\!j_p\!-\!\mu_2}{\tfrac32+\mu_1}{-(\tfrac{z_1}{z_2})^2} \right).
\end{align}
Separating the sum over $N$ according to parities in \eqref{gendev} and substituting \eqref{eNreal} gives
\begin{align*} 
\begin{aligned}
 v_{\lambda_1}^c(z_1,\ &\mu_1,\epsilon_1)v_{\lambda_2}^{\lambda_1}(z_2,\mu_2,\epsilon_2) \\
 &= \sum_{j=0}^\infty K_j(\lambda_2;\mu_2,\mu_1;c) J_j\left( \tfrac{\epsilon_2\lambda_1}{\lambda_2};\ 2\mu_2, 2\mu_1, -\tfrac{c\epsilon_1\epsilon_2}{\lambda_2} \right) \\
 &\times \left( f_e(j) \sum_{k=0}^\infty P_{2k}(\lambda_2;\mu_{12},c\epsilon_{12}) \dfrac{\left(z_1^2+z_2^2\right)^{\frac{2k}{2}}}{[2k]_{\mu_{12}}^{1/2}!} 
  + f_o(j) \sum_{k=0}^\infty P_{2k+1}(\lambda_2;\mu_{12},c\epsilon_{12}) \dfrac{\left(z_1^2+z_2^2\right)^{\frac{2k+1}{2}}}{[2k+1]_{\mu_{12}}^{1/2}!} \right).
\end{aligned}
\end{align*}
The two sums over $k$ have precisely the form of the sums appearing in equation \eqref{splitpar}. It is thus possible to reexpress both of them in terms of an hypergeometric function by using the generating function of the Laguerre polynomials. This yields 
\begin{align*} 
 v_{\lambda_1}^c(z_1, &\mu_1,\epsilon_1)v_{\lambda_2}^{\lambda_1}(z_2,\mu_2,\epsilon_2)\\
 &= \exp\left(-\frac{z_1^2+z_2^2}{2}\right) \sum_{j=0}^\infty K_j(\lambda_2;\mu_2,\mu_1;c) J_j\left( \tfrac{\epsilon_2\lambda_1}{\lambda_2};\ 2\mu_2, 2\mu_1, -\tfrac{c\epsilon_1\epsilon_2}{\lambda_2} \right) \\[0.5em]
 &\times \left( f_e(j)  \pFq{0}{1}{-}{\mu_{12}+\tfrac12}{\tfrac{(z_1^2+z_2^2)(\lambda_2^2-c^2)}{4}} 
 + f_o(j) \frac{(z_1^2+z_2^2)^\frac12 (\lambda_2-c\epsilon_{12})}{2\mu_{12}+1}  \pFq{0}{1}{-}{\mu_{12}+\tfrac32}{\tfrac{(z_1^2+z_2^2)(\lambda_2^2-c^2)}{4}}  \right).
\end{align*}
Using proposition \ref{prop6} to express both eigenvectors on the LHS in terms of special functions, the previous equation becomes a generating function for the Big -1 Jacobi polynomials.

\begin{proposition}
 The Big \m1 Jacobi polynomials satisfy the bilinear generating function
\begin{align*} 
&\left( \pFq{0}{1}{-}{\mu_1+\tfrac12}{\frac{z_1^2 (\lambda_1^2 - c^2)}{4}}  +   \frac{z_1(\lambda_1-c\epsilon_1)}{2\mu+1} \pFq{0}{1}{-}{\mu_1+\tfrac32}{\frac{z_1^2 (\lambda_1^2 - c^2)}{4}} \right) \\[1em]
& \times \left( \pFq{0}{1}{-}{\mu_2+\tfrac12}{\frac{z_2^2 (\lambda_2^2 - \lambda_1^2)}{4}}  +   \frac{z_2(\lambda_2-\lambda_1\epsilon_2)}{2\mu_2+1} \pFq{0}{1}{-}{\mu_2+\tfrac32}{\frac{z_2^2 (\lambda_2^2 - \lambda_1^2)}{4}} \right) \\[1em]
&= \sum_{j=0}^\infty K_j(\lambda_2;\mu_2,\mu_1;c) J_j\left( \tfrac{\epsilon_2\lambda_1}{\lambda_2};\ 2\mu_2, 2\mu_1, -\tfrac{c\epsilon_1\epsilon_2}{\lambda_2} \right) \\[0.5em]
 &\times \left( f_e(j)  \pFq{0}{1}{-}{\mu_{12}+\tfrac12}{\tfrac{(z_1^2+z_2^2)(\lambda_2^2-c^2)}{4}} + f_o(j) \frac{(z_1^2+z_2^2)^\frac12 (\lambda_2-c\epsilon_{12})}{2\mu_{12}+1}  \pFq{0}{1}{-}{\mu_{12}+\tfrac32}{\tfrac{(z_1^2+z_2^2)(\lambda_2^2-c^2)}{4}}  \right)
\end{align*}
where $\mu_{12} = \mu_1+\mu_2+\tfrac12+j$ and $f_e(j)$, $f_o(j)$, $K_j(\lambda_2;\mu_1,\mu_2;c)$ are given by the formulas \eqref{fe}, \eqref{fo}, \eqref{KN}. 
\end{proposition}
\begin{proof}
 The result follows from the analysis provided before the statement of this proposition.
\end{proof}

\section{Conclusion}

We considered the discrete series representations of the superalgebra $\mathfrak{osp}(1|2)$ and singled out a self-adjoint element $X_c$. We constructed the generalized eigenvectors of this special element in the representation spaces and in their two- and three-fold tensor products. Looking at different bases and their overlaps led to our main results : propositions \ref{convid1} and \ref{convid2} which provide convolution identities for \m1 orthogonal polynomials and also connection coefficients for two-variable Dunkl polynomials orthogonal with respect to the same measure. This was further used to obtain a bilinear generating function for the Big \m1 Jacobi polynomials. This led to interpretations and connections between the Specialized Chihara, the dual \m1 Hahn, the Big \m1 Jacobi and the Bannai-Ito polynomials.       

This study suggests a number of future research questions. A natural extension would be to look at higher dimensional spaces via the $n$-fold tensor product of representations. This was done for $\mathfrak{su}(1,1)$ in \cite{Rosengren1999}. This should lead to new convolution identities and to multivariate Dunkl orthogonal polynomials of Tratnik type. In fact, it is straightforward to extend the unitary operators from propositions \ref{prop1}, \ref{prop2} and \ref{prop4} to an arbitrary $n$-fold tensor product; the main difficulty is to find interesting bases and overlaps. It should be noted however that some investigations on this last point have already been done \cite{DeBie2017,DeBie2019}. Another avenue to explore would be how the convolution identities obtained here could be used to derive different generating functions and Poisson kernels. Interesting constructions pointing in this direction have been presented in \cite{VanderJeugt1998} and \cite{Koelink1999} for the Lie algebras $\mathfrak{su}(1,1)$ and $U_q(\mathfrak{su}(1,1))$. Finally, a broader project would be to revisit the construction with different representations and realizations. We mention as examples \cite{Groenevelt2002} and \cite{Groenevelt2004} where similar questions are considered.

\section*{Acknowledgments}
EK thanks the Centre de Recherches Math\'ematiques for its hospitality during a visit when part of this investigation has been done. JML holds an Alexander-Graham-Bell PhD fellowship from the Natural Science and Engineering Research Council (NSERC) of Canada. LV is grateful to NSERC for support through a discovery grant. 


\bibliography{osp_convolution.bib}
\bibliographystyle{elsarticle-num}

\end{document}